\documentclass[a4paper,10pt,english,french]{smfart}
\usepackage[frenchb,english]{babel}
\usepackage[T1]{fontenc}
\usepackage{lmodern}
\usepackage{graphicx}
\usepackage{amsmath,amsthm,amssymb,amsfonts}
\usepackage[a4paper,left=4cm,right=4cm,top=4cm,bottom=4cm]{geometry}
\numberwithin{equation}{section}
\title[Manin's conjecture for a singular quartic del Pezzo surface]{Manin's conjecture for a quartic del Pezzo surface with $\mathbf{A}_3$ singularity and four lines}
\author{Pierre Le Boudec}
\subjclass{$11$D$45$, $14$G$05$}
\address{Université Denis Diderot (Paris VII) \\ Institut de Mathématiques de Jussieu \\ UMR 7586 \\ Case $7012$ - Bâtiment Chevaleret \\ Bureau $7$C$14$ \\ $75205$ Paris Cedex 13, France}
\email{pleboude@math.jussieu.fr}

\begin{document}

\makeatletter
\def\imod#1{\allowbreak\mkern10mu({\operator@font mod}\,\,#1)}
\makeatother

\newtheorem{lemma}{Lemma}
\newtheorem{theorem}{Theorem}
\newtheorem{corollaire}{Corollaire}
\newtheorem{proposition}{Proposition}

\newcommand{\vol}{\operatorname{vol}}
\newcommand{\D}{\mathrm{d}}
\newcommand{\rank}{\operatorname{rank}}
\newcommand{\Pic}{\operatorname{Pic}}
\newcommand{\Gal}{\operatorname{Gal}}
\newcommand{\meas}{\operatorname{meas}}
\newcommand{\Spec}{\operatorname{Spec}}

\begin{abstract}
We establish Manin's conjecture for a quartic del Pezzo surface split over $\mathbb{Q}$ and having a singularity of type $\mathbf{A}_3$ and containing exactly four lines. It is the first example of split singular quartic del Pezzo surface whose universal torsor is not a hypersurface for which Manin's conjecture is proved.
\end{abstract}

\maketitle

\tableofcontents

\section{Introduction}

Manin's conjecture (see \cite{MR974910}) gives a precise description of the distribution of rational points of bounded height on singular del Pezzo surfaces. More precisely, let $V \subset \mathbb{P}^n$ be such a surface defined over $\mathbb{Q}$ and anticanonically embedded and $U$ be the open subset formed by deleting the lines from $V$. We set
\begin{eqnarray*}
N_{U,H}(B) & = & \# \{x \in U(\mathbb{Q}), H(x) \leq B \} \textrm{,}
\end{eqnarray*}
where $H : \mathbb{P}^n(\mathbb{Q}) \to \mathbb{R}_{> 0}$ is the exponential height defined by
\begin{eqnarray*}
H(x_0: \dots :x_n) & = & \max \{ |x_i|, 0 \leq i \leq n \}\textrm{,}
\end{eqnarray*}
for $(x_0, \dots, x_n) \in \mathbb{Z}^{n+1}$ satisfying the condition $\gcd(x_0, \dots, x_n) = 1$. If $\widetilde{V}$ denotes the minimal desingularization of $V$ and $\rho = \rho_{\widetilde{V}}$ the rank of the Picard group of $\widetilde{V}$, then it is expected that
\begin{eqnarray*}
N_{U,H}(B) & = & c_{V,H} B \log(B)^{\rho -1} (1+o(1)) \textrm{,}
\end{eqnarray*}
where $c_{V,H}$ is a constant which is expected to follow Peyre's prediction \cite{MR1340296}. 

We are only interested here in singular del Pezzo surfaces of degree four. Their classification is rather classical and can be found in the work of Coray and Tsfasman \cite{MR940430}. Up to isomorphism over $\overline{\mathbb{Q}}$, there are fifteen types of such surfaces and they are categorized by their extended Dynkin diagrams which are the diagrams describing the intersection behaviour of the negative curves on the minimal desingularizations (see \cite[Table $4$]{D-hyper}). Here is a quick overview of the available results concerning Manin's conjecture for singular quartic del Pezzo surfaces split over $\mathbb{Q}$. The conjecture is already known to hold for nine surfaces of different types. Using harmonic analysis techniques on adelic groups and studying the height Zeta function
\begin{eqnarray*}
Z_{U,H}(s) & = & \sum_{x \in U(\mathbb{Q})} H(x)^{-s} \textrm{,}
\end{eqnarray*}
Batyrev and Tschinkel have proved it for toric varieties \cite{MR1620682} (which covers the three types $4 \mathbf{A}_1$,
$2 \mathbf{A}_1 + \mathbf{A}_2$ and $2 \mathbf{A}_1 + \mathbf{A}_3$) and Chambert-Loir and Tschinkel have proved it for equivariant compactifications of vector groups \cite{MR1906155} (which covers the type $\mathbf{D}_5$). Note that for a certain surface of type $\mathbf{D}_5$, la Bretèche and Browning have proved the conjecture independently \cite{MR2320172}. Finally, the conjecture has been obtained for five other surfaces, a surface of type $\mathbf{D}_4$ by Derenthal and Tschinkel \cite{MR2290499}, a surface of type $\mathbf{A}_1 + \mathbf{A}_3$ by Derenthal \cite{MR2520770}, a surface of type $\mathbf{A}_4$ by Browning and Derenthal \cite{MR2543667} and two surfaces of respective types $3 \mathbf{A}_1$ and $\mathbf{A}_1 + \mathbf{A}_2$ by the author \cite{3A1}. These proofs are very different from those using the fact that the varieties considered are equivariant compactifications of algebraic groups. They all use a lift to universal torsors. This consists in defining a bijection between the set of rational points to be counted on $U$ and a certain set of integral points on an affine variety of higher dimension (which is equal to eight for quartic surfaces). Note that Derenthal has determined the equations of the universal torsors for most of the singular quartic del Pezzo surfaces in his doctoral thesis \cite{Der-th}. This can also be achieved using only elementary techniques, see section \ref{torsor section} for an example.

Our aim is to prove Manin's conjecture for another surface split over $\mathbb{Q}$, having singularity type $\mathbf{A}_3$ and containing exactly four lines. This surface $V \subset \mathbb{P}^4$ is defined as the intersection of the two following quadrics,
\begin{eqnarray*}
x_0 x_1 - x_2^2 & = & 0 \textrm{,} \\
(x_0 + x_1 + x_3) x_3 - x_2 x_4 & = & 0 \textrm{.}
\end{eqnarray*}
The lines on $V$ are given by $x_i = x_2 = x_3 = 0$ and $x_i = x_2 = x_0 + x_1 + x_3 = 0$ for $i \in \{0,1\}$ and the unique singularity is $(0:0:0:0:1)$. We see that $V$ is actually split over $\mathbb{Q}$ and thus, if $\widetilde{V}$ denotes the minimal desingularization of $V$, the Picard group of $\widetilde{V}$ has rank $\rho = 6$. Define the open subset $U$ and the quantity $N_{U,H}(B)$ as explained above. In section \ref{torsor section}, we define a bijection between the set of the points to be counted on $U$ and a certain set of integral points of an open subset of the affine variety embedded in
$\mathbb{A}^{10} \simeq \Spec \left( \mathbb{Q}[\eta_1, \dots, \eta_7, \alpha_1, \alpha_2, \alpha_4] \right)$ and defined by
\begin{eqnarray*}
\eta_1^2\eta_2\eta_4^2\eta_7 + \eta_5 \alpha_1 - \eta_6 \alpha_2 & = & 0 \textrm{,} \\
\eta_2 \eta_3^2 \eta_5^2 \eta_6 + \eta_7 \alpha_2 - \eta_4 \alpha_4 & = & 0 \textrm{.}
\end{eqnarray*}
The universal torsor corresponding to our present problem actually has five equations and can be embedded in
$\mathbb{A}^{11} \simeq \Spec \left( \mathbb{Q}[\eta_1, \dots, \eta_7, \alpha_1, \alpha_2, \alpha_3, \alpha_4] \right)$ but we will neither use these three other equations nor the variable $\alpha_3$. Let us emphasize the fact that it is the first time that Manin's conjecture is proved for a split singular quartic del Pezzo surface whose universal torsor has several equations. This obstacle is overcome in section \ref{First steps} by turning the two equations into a single congruence in order to apply the usual techniques. Our result is the following.

\begin{theorem}
\label{Manin}
As $B$ tends to $+ \infty$, we have the estimate
\begin{eqnarray*}
N_{U,H}(B) & = & c_{V,H} B \log(B)^{5} \left( 1 + O \left( \frac1{\log(B)} \right) \right) \textrm{,}
\end{eqnarray*}
where $c_{V,H}$ agrees with Peyre's prediction.
\end{theorem}

Since $\rho = 6$, this estimate proves that $V$ satisfies Manin's conjecture. Let us note here that Derenthal has proved that $V$ is not toric \cite[Proposition 12]{D-hyper} and Derenthal and Loughran have proved that it is not an equivariant compactification of $\mathbb{G}_a^2$ \cite{DL-equi}, so theorem \ref{Manin} does not follow from the general results \cite{MR1620682} and \cite{MR1906155}. In view of this result, it only remains to deal with five types of split singular quartic del Pezzo surfaces among the list of fifteen.

In the following section, we prove several lemmas about summations of arithmetic functions. The next two sections are respectively devoted to the calculations of the universal torsor and of Peyre's constant. Finally, the last section is dedicated to the proof of theorem \ref{Manin}.

It is a great pleasure for the author to thank his supervisor Professor de la Bretèche both for his encouragement and his advice during this work.

This work has received the financial support of the ANR PEPR (Points Entiers Points Rationnels).

\section{Arithmetic functions}

We need to introduce the following collection of arithmetic functions,
\begin{eqnarray*}
\varphi^{\ast}(n) = \prod_{p|n} \left( 1 - \frac1{p} \right) \textrm{,} & \ &
\varphi^{\circ}(n) = \prod_{\substack{p|n \\ p \neq 2}} \left( 1 - \frac1{p-1} \right) \textrm{,} \\
\varphi^{\dag}(n) =  \prod_{p|n} \left( 1 - \frac1{p^2} \right) \textrm{,} & \ &
\varphi^{\flat}(n) = \prod_{\substack{p|n \\ p \neq 2}} \left( 1 + \frac1{p(p-2)} \right) \textrm{.}
\end{eqnarray*}
We can note here that if $n$ is odd then $\varphi^{\circ}(n) \varphi^{\flat}(n) = \varphi^{\ast}(n)$ and if $n$ is even then
$\varphi^{\circ}(n) \varphi^{\flat}(n) = 2 \varphi^{\ast}(n)$. Moreover, for $a,b \geq 1$, we define
\begin{eqnarray*}
\psi_{a,b}(n) & = &
\begin{cases}
\varphi^{\circ}(\gcd(a,n))^{-1} & \textrm{ if } \gcd(n,b) = 1 \textrm{,} \\
0 & \textrm{ otherwise,}
\end{cases}
\end{eqnarray*}
and
\begin{eqnarray*}
\psi_{a,b}'(n) & = &
\begin{cases}
\varphi^{\circ}(\gcd(a,n))^{-1} \varphi^{\ast}(n) \varphi^{\ast}(\gcd(a,n))^{-1} & \textrm{ if } \gcd(n,b) = 1 \textrm{,} \\
0 & \textrm{ otherwise.}
\end{cases}
\end{eqnarray*}
Finally, for $\delta > 0$, we set
\begin{eqnarray*}
\sigma_{- \delta}(n) & = & \sum_{k|n} k^{- \delta} \textrm{.}
\end{eqnarray*}

\begin{lemma}
\label{arithmetic preliminary 0}
Let $0 < \delta \leq 1$ be fixed. We have the estimate
\begin{eqnarray*}
\sum_{n \leq X} \psi_{a,b}(n)& = & \Psi(a,b) X + O_{\delta} \left( \sigma_{- \delta}(ab) X^{\delta} \right) \textrm{,}
\end{eqnarray*}
where
\begin{eqnarray*}
\Psi(a,b) & = & \varphi^{\ast}(b) \frac{\varphi^{\flat}(a)}{\varphi^{\flat}(\gcd(a,b))} \textrm{.}
\end{eqnarray*}
\end{lemma}

\begin{proof}
We start by calculating the Dirichlet convolution of $\psi_{a,b}$ with the Möbius function $\mu$. We have
\begin{eqnarray*}
(\psi_{a,b} \ast \mu)(n)& = & \sum_{d|n} \psi_{a,b} \left( \frac{n}{d} \right) \mu(d) \\
& = & \prod_{p^\nu \parallel n} \left( \psi_{a,b} \left( p^\nu \right) - \psi_{a,b} \left( p^{\nu - 1} \right) \right) \textrm{.}
\end{eqnarray*}
Moreover $\psi_{a,b}(1) = 1$ and for all $\nu \geq 1$, we have
\begin{eqnarray*}
\psi_{a,b} \left( p^\nu \right) = \psi_{a,b}(p) =
\begin{cases}
\left( 1 - 1/(p-1) \right)^{-1} & \textrm{ if } p|a, p \neq 2 \textrm{ and } p \nmid b \textrm{,} \\
1 & \textrm{ if } p \neq 2, p \nmid ab \textrm{,} \\
1 & \textrm{ if } p = 2, 2 \nmid b \textrm{,} \\
0 & \textrm{ if } p|b \textrm{.}
\end{cases}
\end{eqnarray*}
Thus, we easily obtain
\begin{eqnarray*}
(\psi_{a,b} \ast \mu)(n) & = &
\mu(n) \prod_{p|\gcd(a,n), p\nmid b} \frac{-1}{p-2} \textrm{,}
\end{eqnarray*}
if $n|ab$ and $2 \nmid n$ or $2|b$ and $(\psi_{a,b} \ast \mu)(n) = 0$ otherwise. Writing $\psi_{a,b} = (\psi_{a,b} \ast \mu) \ast 1$, we get
\begin{eqnarray*}
\sum_{n \leq X} \psi_{a,b}(n) & = & \sum_{n \leq X} \sum_{d|n} (\psi_{a,b} \ast \mu)(d) \\
& = & \sum_{d = 1}^{+ \infty} (\psi_{a,b} \ast \mu)(d) \left[ \frac{X}{d} \right] \textrm{.}
\end{eqnarray*}
Let $0 < \delta \leq 1$ be fixed. Let us use the elementary estimate $[t] = t + O \left( t^{\delta} \right)$ for $t = X/d$. Since
$|(\psi_{a,b} \ast \mu)(n)| \leq 1$, we get
\begin{eqnarray*}
\sum_{d = 1}^{+ \infty} \frac{| (\psi_{a,b}  \ast \mu)(d) |}{d^{\delta}} & \leq & \sigma_{- \delta}(ab) \textrm{,}
\end{eqnarray*}
and we have thus proved that
\begin{eqnarray*}
\sum_{n \leq X} \psi_{a,b}(n) & = & X \sum_{d = 1}^{+ \infty} \frac{(\psi_{a,b} \ast \mu)(d)}{d}
+ O \left( \sigma_{- \delta}(ab) X^{\delta} \right) \textrm{.}
\end{eqnarray*}
Finally, a straigthforward calculation gives
\begin{eqnarray*}
\sum_{d = 1}^{+ \infty} \frac{(\psi_{a,b}  \ast \mu)(d)}{d} & = & \prod_{p|b} \left( 1 - \frac1{p} \right)
\prod_{\substack{p|a, p \nmid b \\ p \neq 2}} \left( 1 + \frac1{p(p-2)} \right) \textrm{,}
\end{eqnarray*}
which concludes the proof.
\end{proof}

\begin{lemma}
\label{arithmetic preliminary 0'}
Let $0 < \delta \leq 1$ be fixed. We have the estimate
\begin{eqnarray*}
\sum_{n \leq X} \psi_{a,b}'(n)& = & \Psi'(a,b) X + O_{\delta} \left( \sigma_{- \delta}(b) X^{\delta} \right) \textrm{,}
\end{eqnarray*}
where
\begin{eqnarray*}
\Psi'(a,b) & = & \varphi^{\ast}(b) \frac{\varphi^{\flat}(a)}{\varphi^{\flat}(\gcd(a,b))} \frac{\zeta(2)^{-1}}{\varphi^{\dag}(ab)} \textrm{.}
\end{eqnarray*}
\end{lemma}

\begin{proof}
We proceed exactly as for the proof of lemma \ref{arithmetic preliminary 0}. Let
\begin{eqnarray*}
f(n) & = & \mu(n) \prod_{\substack{p|n, p\nmid ab \\ p \neq 2}} \frac1{p}
\prod_{\substack{p|\gcd(a,n), p\nmid b \\ p \neq 2}} \frac{-1}{p-2} \textrm{.}
\end{eqnarray*}
A calculation provides
\begin{eqnarray*}
(\psi_{a,b}' \ast \mu)(n) & = &
\begin{cases}
f(n) & \textrm{ if } 2 \nmid n \textrm{ or } 2|b \textrm{,} \\
f(n)/2 & \textrm{ if } 2|n \textrm{ and } 2 \nmid ab \textrm{,} \\
0 & \textrm{ otherwise.}
\end{cases}
\end{eqnarray*}
Now we see that $| (\psi_{a,b}' \ast \mu)(n) | \ll \gcd(b,n)/n$, which easily yields
\begin{eqnarray*}
\sum_{d = 1}^{+ \infty} \frac{| (\psi_{a,b}'  \ast \mu)(d) |}{d^{\delta}} & \ll & \sigma_{- \delta}(b) \textrm{.}
\end{eqnarray*}
Another straightforward calculation gives
\begin{eqnarray*}
\sum_{d = 1}^{+ \infty} \frac{(\psi_{a,b}' \ast \mu)(d)}{d} & = & \Psi'(a,b) \textrm{,}
\end{eqnarray*}
which completes the proof.
\end{proof}

Using partial summation and the estimates of lemmas \ref{arithmetic preliminary 0} and \ref{arithmetic preliminary 0'} as in the proof of \cite[Lemma $6$]{3A1}, we see that we have the following result.

\begin{lemma}
\label{arithmetic preliminary}
Let $0 < \delta \leq 1$ be fixed. Let $0 \leq t_1 < t_2$ and $I=[t_1,t_2]$. Let also $g : \mathbb{R}_{> 0} \to \mathbb{R}$ be a function having a piecewise continuous derivative on $I$ whose sign changes at most $R_g(I)$ times on $I$. We have
\begin{eqnarray*}
\sum_{n \in I \cap \mathbb{Z}_{>0}} \psi_{a,b}(n) g(n) & = & \Psi(a,b) \int_I g(t) \D t +
O_{\delta} \left( \sigma_{- \delta}(ab) t_2^{\delta} M_I(g) \right) \textrm{,}
\end{eqnarray*}
and 
\begin{eqnarray*}
\sum_{n \in I \cap \mathbb{Z}_{>0}} \psi_{a,b}'(n) g(n) & = & \Psi'(a,b) \int_I g(t) \D t +
O_{\delta} \left( \sigma_{- \delta}(b) t_2^{\delta} M_I(g) \right) \textrm{,}
\end{eqnarray*}
where $M_I(g) = (1 + R_g(I)) \sup_{t \in I \cap \mathbb{R}_{> 0}} |g(t)|$.
\end{lemma}

We also have the following estimation.

\begin{lemma}
\label{arithmetic preliminary 2}
With the same notations, if $2 \nmid b$ then
\begin{eqnarray*}
\sum_{\substack{n \in I \cap \mathbb{Z}_{>0} \\ n \equiv 0 \imod{2}}} \psi_{a,b}(n) g(n) & = &
\frac1{2} \Psi(a,b) \int_I g(t) \D t + O_{\delta} \left( \sigma_{- \delta}(ab) t_2^{\delta} M_I(g) \right) \textrm{.}
\end{eqnarray*}
In a similar way, if $2|a$ and $2 \nmid b$ then
\begin{eqnarray*}
\sum_{\substack{n \in I \cap \mathbb{Z}_{>0} \\ n \equiv 0 \imod{2}}} \psi_{a,b}'(n) g(n) & = &
\frac1{2} \Psi'(a,b) \int_I g(t) \D t + O_{\delta} \left( \sigma_{- \delta}(b) t_2^{\delta} M_I(g) \right) \textrm{.}
\end{eqnarray*}
\end{lemma}

\begin{proof}
Let us prove the statement for $\psi_{a,b}$, it suffices to notice that
\begin{eqnarray*}
\sum_{\substack{n \leq X \\ n \equiv 0 \imod{2}}} \psi_{a,b}(n) & = &
\sum_{d = 1}^{+ \infty} (\psi_{a,b} \ast \mu)(d) \sum_{\substack{k \leq X/d \\ k \equiv 0 \imod{2}}} 1 \\
& & + \sum_{\substack{d = 1 \\ d \equiv 0 \imod{2}}}^{+ \infty} (\psi_{a,b} \ast \mu)(d)
\sum_{\substack{k \leq X/d \\ k \equiv 1 \imod{2}}} 1 \textrm{,}
\end{eqnarray*}
and $(\psi_{a,b} \ast \mu)(d) = 0$ for all $d \equiv 0 \imod{2}$ since $2 \nmid b$ and therefore
\begin{eqnarray*}
\sum_{\substack{n \leq X \\ n \equiv 0 \imod{2}}} \psi_{a,b}(n) & = & \sum_{d = 1}^{+ \infty} (\psi_{a,b} \ast \mu)(d) \left( \frac{X}{2 d}
+ O \left( \frac{X^{\delta}}{d^{\delta}} \right) \right) \textrm{.}
\end{eqnarray*}
We can conclude exactly as in the proof of lemma \ref{arithmetic preliminary 0} and finally, as for lemma
\ref{arithmetic preliminary}, use partial summation to complete the proof. The proof for $\psi_{a,b}'$ is strictly identical, it only uses the fact that $(\psi_{a,b}' \ast \mu)(d) = 0$ for all $d \equiv 0 \imod{2}$ since $2|a$ and $2 \nmid b$.
\end{proof}

\section{The universal torsor}

\label{torsor section}

We now proceed to define a bijection between the set of rational points we want to count on $U$ and a certain set of integral points on the affine variety defined in the introduction. As explained in the introduction, the universal torsor of our problem is an open subset of an affine variety of dimension $8$ embedded in $\mathbb{A}^{11}$. It has five equations but we will only deal with ten of the eleven variables and will only make use of two equations among these five. Our choice of notation might be surprising but it is guided by our wish to adopt the notation used by Derenthal in \cite[Chapter $6$]{Der-th}. Note that if $(x_0:x_1:x_2:x_3:x_4) \in V(\mathbb{Q})$ then we have $(x_0:x_1:x_2:x_3:x_4) \in U(\mathbb{Q})$ if and only if $x_0x_1x_2x_3 \neq 0$. Let
$(x_0,x_1,x_2,x_3,x_4) \in \mathbb{Z}_{\neq 0}^4 \times \mathbb{Z}$ be such that
\begin{eqnarray*}
x_0 x_1 - x_2^2 & = & 0 \textrm{,} \\
(x_0 + x_1 + x_3) x_3 - x_2 x_4 & = & 0 \textrm{,}
\end{eqnarray*}
and $\max \{ |x_i|, 0 \leq i \leq 4 \} \leq B$ and $\gcd(x_0,x_1,x_2,x_3,x_4) = 1$. Since $\mathbf{x} = - \mathbf{x}$ in $\mathbb{P}^4$, we can assume that $x_0 > 0$, which implies $x_1 > 0$. Moreover, the symmetry given by $(x_2,x_4) \mapsto (-x_2,-x_4)$ shows that we can also assume that $x_2 > 0$ keeping in mind that we need to multiply our future result by $2$. The first equation shows that there is a unique way to write $x_0 = y_{01} x_0'^2$, $x_1 = y_{01} x_1'^2$ and $x_2 = y_{01} x_0' x_1'$ for some $x_0',x_1',y_{01} > 0$ such that $\gcd(x_0',x_1') = 1$. The second equation therefore gives
\begin{eqnarray*}
\left( y_{01} x_0'^2+ y_{01} x_1'^2 + x_3 \right) x_3 - y_{01} x_0' x_1' x_4 & = & 0 \textrm{.}
\end{eqnarray*}
We define $y_{01}' = \gcd(y_{01}, x_3) > 0$ and write $y_{01} = y_{01}' \eta_2$ and $x_3 =  y_{01}' x_3'$ with $\eta_2>0$ and $\gcd(\eta_2, x_3') = 1$. We obtain
\begin{eqnarray*}
\left( \eta_2 x_0'^2+ \eta_2 x_1'^2 + x_3' \right) y_{01}' x_3' - \eta_2 x_0' x_1' x_4 & = & 0 \textrm{,}
\end{eqnarray*}
and thus $\eta_2|y_{01}' x_3'^2$ and it follows $\eta_2|y_{01}'$ since $\gcd(\eta_2,x_3') = 1$. We can therefore write $y_{01}' = \eta_2 y_{01}''$ for some $y_{01}'' > 0$. The equation becomes
\begin{eqnarray*}
\left( \eta_2 x_0'^2+ \eta_2 x_1'^2 + x_3' \right) y_{01}'' x_3' - x_0' x_1' x_4 & = & 0 \textrm{.}
\end{eqnarray*}
We now see that $\gcd(x_0,x_1,x_2,x_3,x_4) = 1$ implies $\gcd(y_{01}'',x_4) = 1$ and thus $y_{01}''|x_0' x_1'$ and $x_0'$, $x_1'$ being coprime, we can write $y_{01}'' = \eta_1 \eta_3$, $x_0' = \eta_3 x_0''$ and $x_1' = \eta_1 x_1''$ for some
$\eta_1, \eta_3, x_0'', x_1'' > 0$. Now we set
$x_3' = \alpha_1 x_3''$, $x_4 = \alpha_1 \alpha_4$ with $x_3'' > 0$ and $\gcd(x_3'',\alpha_4) = 1$ (we do not prescribe the sign of
$\alpha_1=\pm \gcd(x_3',x_4)$). We finally get
\begin{eqnarray*}
\left( \eta_2 \eta_3^2 x_0''^2 + \eta_2 \eta_1^2 x_1''^2 + \alpha_1 x_3'' \right) x_3'' - x_0'' x_1'' \alpha_4 & = & 0 \textrm{.}
\end{eqnarray*}
We observe that since $\gcd(x_3'',\alpha_4) = 1$, we have $x_3''|x_0''x_1''$ and we can write $x_3'' = \eta_5 \eta_7$, $x_0'' = \eta_5 \eta_6$ and $x_1'' = \eta_4 \eta_7$, for some $\eta_4,\eta_5,\eta_6,\eta_7 > 0$. We have finally obtained
\begin{eqnarray*}
x_0 & = & \eta_1 \eta_2^2 \eta_3^3 \eta_5^2 \eta_6^2 \textrm{,} \\
x_1 & = & \eta_1^3 \eta_2^2 \eta_3 \eta_4^2 \eta_7^2 \textrm{,} \\
x_2 & = & \eta_1^2 \eta_2^2 \eta_3^2 \eta_4 \eta_5 \eta_6 \eta_7 \textrm{,} \\
x_3 & = & \eta_1 \eta_2 \eta_3 \eta_5 \eta_7 \alpha_1 \textrm{,} \\
x_4 & = & \alpha_1 \alpha_4 \textrm{,}
\end{eqnarray*}
and the equation is
\begin{eqnarray*}
\eta_2 \eta_3^2 \eta_5^2 \eta_6^2 + \eta_1^2 \eta_2 \eta_4^2 \eta_7^2 + \eta_5 \eta_7 \alpha_1 -
\eta_4 \eta_6 \alpha_4 & = & 0 \textrm{.}
\end{eqnarray*}
Furthermore, it is easy to see that the coprimality conditions can be summed up by
\begin{eqnarray}
\label{coprim1}
& & \gcd(\eta_3\eta_5\eta_6,\eta_1\eta_4\eta_7) = 1 \textrm{,} \\
\label{coprim2}
& & \gcd(\eta_5\eta_7,\eta_2\alpha_4) = 1 \textrm{,} \\
\label{coprim3}
& & \gcd(\eta_1\eta_2\eta_3, \alpha_1\alpha_4) = 1 \textrm{.}
\end{eqnarray}
Since $\eta_6$ and $\eta_7$ are coprime, we see that the equation is equivalent to the existence of
$\alpha_2 \in \mathbb{Z}$ such that
\begin{eqnarray}
\label{torsor 1}
\eta_1^2 \eta_2 \eta_4^2 \eta_7 + \eta_5 \alpha_1 - \eta_6 \alpha_2 & = & 0 \textrm{,} \\
\label{torsor 2}
\eta_2 \eta_3^2 \eta_5^2 \eta_6 + \eta_7 \alpha_2 - \eta_4 \alpha_4 & = & 0 \textrm{.}
\end{eqnarray}
In a similar way, since $\eta_4$ and $\eta_5$ are coprime, we can derive the existence of $\alpha_3 \in \mathbb{Z}$ such that
\begin{eqnarray*}
\eta_2 \eta_3^2 \eta_5 \eta_6^2 + \eta_7 \alpha_1 - \eta_4 \alpha_3 & = & 0 \textrm{,} \\
\eta_1^2 \eta_2 \eta_4 \eta_7^2 + \eta_5 \alpha_3 - \eta_6 \alpha_4 & = & 0 \textrm{,} \\
\eta_1^2 \eta_2^2 \eta_3^2 \eta_4 \eta_5 \eta_6 \eta_7 + \alpha_1 \alpha_4 - \alpha_2 \alpha_3 & = & 0 \textrm{.}
\end{eqnarray*}
As explained above, we will not use these three equations. We define $\mathcal{T}(B)$ as the set of
$(\eta_1,\eta_2,\eta_3,\eta_4,\eta_5,\eta_6,\eta_7,\alpha_1,\alpha_2,\alpha_4) \in \mathbb{Z}_{>0}^7 \times \mathbb{Z}^3$ satisfying the coprimality conditions \eqref{coprim1}, \eqref{coprim2}, \eqref{coprim3}, the two equations \eqref{torsor 1} and \eqref{torsor 2} and finally the height conditions
\begin{eqnarray}
\label{condition1}
\eta_1 \eta_2^2 \eta_3^3 \eta_5^2 \eta_6^2 & \leq & B \textrm{,} \\
\label{condition2}
\eta_1^3 \eta_2^2 \eta_3 \eta_4^2 \eta_7^2 & \leq & B \textrm{,} \\
\label{condition3}
\eta_1 \eta_2 \eta_3 \eta_5 \eta_7 |\alpha_1| & \leq & B \textrm{,} \\
\label{condition4}
|\alpha_1 \alpha_4| & \leq & B \textrm{.}
\end{eqnarray}
We have proved the following lemma.

\begin{lemma}
\label{T}
We have the equality
\begin{eqnarray*}
N_{U,H}(B) & = & 2 \# \mathcal{T}(B) \textrm{.}
\end{eqnarray*}
\end{lemma}

\section{Calculation of Peyre's constant}

We calculate the value of the constant $c_{V,H}$ predicted by Peyre. It is defined by
\begin{eqnarray*}
c_{V,H} & = & \alpha(\widetilde{V}) \beta(\widetilde{V}) \omega_H(\widetilde{V}) \textrm{,}
\end{eqnarray*}
where $\alpha(\widetilde{V}) \in \mathbb{Q}$ is the volume of a certain polytope in the dual of the effective cone of $\widetilde{V}$ with respect to the intersection form,
$\beta(\widetilde{V}) = \# H^1(\Gal(\overline{\mathbb{Q}}/\mathbb{Q}), \Pic_{\overline{\mathbb{Q}}}(\widetilde{V})) = 1$ since $V$ is split over $\mathbb{Q}$ and finally
\begin{eqnarray*}
\omega_H(\widetilde{V}) & = & \omega_{\infty} \prod_p \left( 1 - \frac1{p} \right)^6 \omega_p  \textrm{,}
\end{eqnarray*}
where $\omega_{\infty}$ and $\omega_p$ are respectively the archimedean and $p$-adic densities. The work of Derenthal \cite{MR2318651} provides the value
\begin{eqnarray*}
\alpha(\widetilde{V}) & = & \frac1{4320} \textrm{.}
\end{eqnarray*}
Furthermore, using \cite[Lemma 2.3]{Loughran}, we get
\begin{eqnarray*}
\omega_p & = & 1 + \frac{6}{p} + \frac1{p^2} \textrm{.}
\end{eqnarray*}
To calculate $\omega_{\infty}$, we set $f_1(x) = x_0 x_1 - x_2^2$, $f_2(x) = (x_0 + x_1 + x_3) x_3 - x_2 x_4$ and we parametrize the points of $V$ by $x_0$, $x_2$ and $x_3$. We have
\begin{eqnarray*}
\det \begin{pmatrix}
\frac{\partial f_1}{\partial x_1} & \frac{\partial f_1}{\partial x_4} \\
\frac{\partial f_2}{\partial x_1} & \frac{\partial f_2}{\partial x_4}
\end{pmatrix} & = &
\begin{vmatrix}
x_0 & 0 \\
x_3 & -x_2
\end{vmatrix} \\
& = & - x_0 x_2 \textrm{.}
\end{eqnarray*}
Moreover, $x_1 = x_2^2/x_0$ and $x_4 = (x_0^2 + x_2^2 + x_0x_3) x_3 / ( x_0 x_2 )$. Since
$\mathbf{x} = - \mathbf{x}$ in $\mathbb{P}^4$, we have
\begin{eqnarray*}
\omega_{\infty} & = & 2 \int \int \int_{x_0, x_2 > 0, x_0, x_2^2/x_0, |x_3|,
\left|x_0^2 + x_2^2 + x_0x_3 \right| \left| x_3 \right| / x_0 x_2 \leq 1}
\frac{\D x_0 \D x_2 \D x_3}{x_0x_2} \textrm{.}
\end{eqnarray*}
Define the function
\begin{eqnarray}
\label{equation h}
h & : & (u_2,t_7,t_6) \mapsto \max \{t_6,t_7, t_7 |t_7 - t_6 u_2|, |t_7 - t_6 u_2| |t_6 + t_7 u_2| \} \textrm{.}
\end{eqnarray}
The change of variables given by $x_0 = t_6^2$, $x_2 = t_6 t_7$ and $x_3 = - t_7 (t_7 - t_6 u_2)$ yields
\begin{eqnarray}
\label{omega}
\omega_{\infty} & = & 4 \int \int \int_{t_6,t_7>0, h(u_2,t_7,t_6) \leq 1} \D u_2 \D t_7 \D t_6 \textrm{.}
\end{eqnarray}

\section{Proof of the main theorem}

\subsection{First steps of the proof}

\label{First steps}

The idea of the proof is to see the equations \eqref{torsor 1} and \eqref{torsor 2} as congruences respectively modulo $\eta_5$ and $\eta_4$ and then to count the number of $\alpha_2$ satisfying these two congruences. In order to do so, we replace the height conditions \eqref{condition3} and \eqref{condition4} by
\begin{eqnarray*}
\eta_1 \eta_2 \eta_3 \eta_7 \left| \eta_1^2 \eta_2 \eta_4^2 \eta_7 - \eta_6 \alpha_2 \right| & \leq & B \textrm{,} \\
\eta_4^{-1} \eta_5^{-1} \left| \eta_1^2 \eta_2 \eta_4^2 \eta_7 - \eta_6 \alpha_2 \right|
\left| \eta_2 \eta_3^2 \eta_5^2 \eta_6 + \eta_7 \alpha_2 \right| & \leq & B \textrm{,}
\end{eqnarray*}
and we carry on denoting them the same way. We note that the equation \eqref{torsor 1} proves that we necessarily have
$\gcd(\eta_1\eta_2,\eta_6\alpha_2) = 1$ since we also have $\gcd(\eta_1\eta_2,\eta_5\alpha_1)~=~1$. Exactly the same way we get
$\gcd(\alpha_2,\eta_3\eta_5) = 1$ thanks to the equation \eqref{torsor 2} and $\gcd(\eta_3\eta_5,\eta_4\alpha_4) = 1$. The equation \eqref{torsor 2} and $\gcd(\eta_2,\eta_7\alpha_2) = 1$ also imply $\gcd(\eta_2,\eta_4) = 1$. This new coprimality condition together with the equation \eqref{torsor 2} yield $\gcd(\eta_4,\alpha_2) = 1$ since we have $\gcd(\eta_4,\eta_2\eta_3\eta_5\eta_6) = 1$. In a similar way, we finally obtain $\gcd(\alpha_1,\eta_4\eta_6) = 1$, $\gcd(\eta_4,\eta_7) = 1$ and $\gcd(\eta_5,\eta_6) = 1$. We can therefore rewrite the coprimality conditions \eqref{coprim1}, \eqref{coprim2}, \eqref{coprim3} and all these new conditions as
\begin{eqnarray}
\label{gcd1}
& & \gcd(\alpha_1,\eta_1\eta_2\eta_3\eta_4\eta_6) = 1 \textrm{,} \\
\label{gcd2}
& & \gcd(\alpha_4,\eta_1\eta_2\eta_3\eta_5\eta_7) = 1 \textrm{,} \\
\label{gcd3}
& & \gcd(\alpha_2,\eta_1\eta_2\eta_3\eta_4\eta_5) = 1 \textrm{,} \\
\label{gcd4}
& & \gcd(\eta_7,\eta_2\eta_3\eta_4\eta_5\eta_6) = 1 \textrm{,} \\
\label{gcd5}
& & \gcd(\eta_6,\eta_1\eta_2\eta_4\eta_5) = 1 \textrm{,} \\
\label{gcd6}
& & \gcd(\eta_1\eta_4,\eta_3\eta_5) = 1 \textrm{,} \\
\label{gcd7}
& & \gcd(\eta_2,\eta_4\eta_5) = 1 \textrm{.}
\end{eqnarray}
From now on, we set $\boldsymbol{\eta} = (\eta_1, \eta_2, \eta_3, \eta_4, \eta_5) \in \mathbb{Z}_{>0}^5$ and
$\boldsymbol{\eta}' = (\boldsymbol{\eta},\eta_6,\eta_7) \in \mathbb{Z}_{>0}^7$. Consider that
$\boldsymbol{\eta}' \in \mathbb{Z}_{>0}^7$ is fixed and is subject to the height conditions \eqref{condition1}, \eqref{condition2} and to the coprimality conditions \eqref{gcd4}, \eqref{gcd5}, \eqref{gcd6} and \eqref{gcd7}. Let $N(\boldsymbol{\eta}',B)$ be the number of $(\alpha_1,\alpha_2,\alpha_4) \in \mathbb{Z}$ satisfying the equations \eqref{torsor 1}, \eqref{torsor 2}, the height conditions \eqref{condition3} and \eqref{condition4} and finally the coprimality conditions \eqref{gcd1}, \eqref{gcd2} and \eqref{gcd3}. For $(r_1,r_2,r_3,r_4,r_5) \in \mathbb{Q}^5$, we define
\begin{eqnarray*}
\boldsymbol{\eta}^{(r_1,r_2,r_3,r_4,r_5)} & = & \eta_1^{r_1} \eta_2^{r_2} \eta_3^{r_3} \eta_4^{r_4} \eta_5^{r_5} \textrm{,}
\end{eqnarray*}
and we adopt the following notations in order to help in the understanding of the height conditions,
\begin{eqnarray*}
A_2 & = & \boldsymbol{\eta}^{(1,1,1,1,1)} \textrm{,} \\
Y_6 & = & \frac{B^{1/2}}{\boldsymbol{\eta}^{(1/2,1,3/2,0,1)}} \textrm{,} \\
Y_7 & = & \frac{B^{1/2}}{\boldsymbol{\eta}^{(3/2,1,1/2,1,0)}} \textrm{,}
\end{eqnarray*}
and recalling the definition \eqref{equation h} of the function $h$, we can sum up the height conditions \eqref{condition1}, \eqref{condition2}, \eqref{condition3} and \eqref{condition4} as
\begin{eqnarray*}
h \left( \frac{\alpha_2}{A_2}, \frac{\eta_7}{Y_7}, \frac{\eta_6}{Y_6} \right) & \leq & 1 \textrm{.}
\end{eqnarray*}
We also introduce the real-valued functions
\begin{eqnarray*}
g_1 & : & (t_7,t_6) \mapsto \int_{h(u_2,t_7,t_6) \leq 1} \D u_2 \textrm{,} \\
g_2 & : & (t_6;\boldsymbol{\eta},B) \mapsto \int_{t_7 Y_7 \geq 1} g_1(t_7,t_6) \D t_7 \textrm{,} \\
g_3 & : & (\boldsymbol{\eta},B) \mapsto \int_{t_6 Y_6 \geq 1} g_2(t_6;\boldsymbol{\eta},B) \D t_6 \textrm{.}
\end{eqnarray*}
We obviously have
\begin{eqnarray}
\label{g_3}
g_3(\boldsymbol{\eta},B) & = & \int \int \int_{t_6 Y_6 \geq 1, t_7 Y_7 \geq 1, h(u_2,t_7,t_6) \leq 1} \D u_2 \D t_7 \D t_6 \textrm{.}
\end{eqnarray}

\begin{lemma}
\label{bounds}
We have the bounds
\begin{eqnarray*}
g_1(t_7,t_6) & \ll & t_6^{-1/2} t_7^{-1/2} \textrm{,} \\
g_2(t_6;\boldsymbol{\eta},B) & \ll & t_6^{-1/2} \textrm{.}
\end{eqnarray*}
\end{lemma}

\begin{proof}
Recall the definition \eqref{equation h} of the function $h$. A little thought reveals that the condition $|t_7 - t_6 u_2| |t_6 + t_7 u_2| \leq 1$ implies that $u_2$ runs over a set whose measure is $\ll t_6^{-1/2} t_7^{-1/2}$ which gives the first bound. The second bound is an immediate consequence of the first since
$t_7 \leq 1$.
\end{proof}

We have the following result.

\begin{lemma}
\label{lemma inter}
The following estimate holds
\begin{eqnarray*}
N(\boldsymbol{\eta}',B) & = & \frac{A_2}{\eta_4 \eta_5} g_1 \left( \frac{\eta_7}{Y_7}, \frac{\eta_6}{Y_6} \right) \theta(\boldsymbol{\eta}') +
R( \boldsymbol{\eta}', B) \textrm{,}
\end{eqnarray*}
where $\theta(\boldsymbol{\eta}')$ is a certain arithmetic function given in \eqref{theta} and
\begin{eqnarray*}
\sum_{\boldsymbol{\eta}'} R( \boldsymbol{\eta}', B) & \ll & B \log(B)^2 \textrm{.}
\end{eqnarray*}
\end{lemma}

Let us remove the coprimality conditions \eqref{gcd1} and \eqref{gcd2} employing two Möbius inversions, we get
\begin{eqnarray*}
N(\boldsymbol{\eta}',B) & = & \sum_{k_1|\eta_1\eta_2\eta_3\eta_4\eta_6} \mu(k_1) \sum_{k_4|\eta_1\eta_2\eta_3\eta_5\eta_7} \mu(k_4)
S_{k_1,k_4} \textrm{,}
\end{eqnarray*}
where, with the notations $\alpha_1 = k_1 \alpha_1'$ and $\alpha_4 = k_4 \alpha_4'$,
\begin{eqnarray*}
S_{k_1,k_4} & = & \# \left\{ (\alpha_1',\alpha_4',\alpha_2) \in \mathbb{Z}^3,
\begin{array}{l}
\eta_1^2 \eta_2 \eta_4^2 \eta_7 + \eta_5 k_1 \alpha_1' - \eta_6 \alpha_2 = 0 \\
\eta_2 \eta_3^2 \eta_5^2 \eta_6 + \eta_7 \alpha_2 - \eta_4 k_4 \alpha_4' = 0 \\
\eqref{condition3}, \eqref{condition4}, \eqref{gcd3}
\end{array}
\right\} \\
& = & \# \left\{ \alpha_2 \in \mathbb{Z},
\begin{array}{l}
\eta_6 \alpha_2 \equiv \eta_1^2 \eta_2 \eta_4^2 \eta_7 \imod{k_1 \eta_5} \\
\eta_7 \alpha_2 \equiv - \eta_2 \eta_3^2 \eta_5^2 \eta_6 \imod{k_4 \eta_4} \\
\eqref{condition3}, \eqref{condition4}, \eqref{gcd3}
\end{array}
\right\} \textrm{.}
\end{eqnarray*}
We note that we necessarily have $\gcd(k_1,\eta_6) = 1$ since $\gcd(\eta_6,\eta_1\eta_2\eta_4\eta_7) = 1$ and $\gcd(k_1,\eta_1\eta_2\eta_4) = 1$ since $\gcd(\eta_1\eta_2\eta_4,\eta_6\alpha_2) = 1$. In a similar way, we also have $\gcd(k_4,\eta_2\eta_3\eta_5\eta_7)=1$. In particular, we see that $\eta_6$ and $\eta_7$ are respectively invertible modulo $k_1 \eta_5$  and $k_4 \eta_4$. We therefore get
\begin{eqnarray*}
N(\boldsymbol{\eta}',B) & = & \sum_{\substack{k_1|\eta_3 \\ \gcd(k_1,\eta_1\eta_2\eta_4\eta_6) = 1}} \mu(k_1)
\sum_{\substack{k_4|\eta_1 \\ \gcd(k_4,\eta_2\eta_3\eta_5\eta_7) = 1}} \mu(k_4) S_{k_1,k_4} \textrm{,}
\end{eqnarray*}
and
\begin{eqnarray*}
S_{k_1,k_4} & = & \# \left\{ \alpha_2 \in \mathbb{Z},
\begin{array}{l}
\alpha_2 \equiv \eta_6^{-1} \eta_1^2 \eta_2 \eta_4^2 \eta_7 \imod{k_1 \eta_5} \\
\alpha_2 \equiv - \eta_7^{-1} \eta_2 \eta_3^2 \eta_5^2 \eta_6 \imod{k_4 \eta_4} \\
\eqref{condition3}, \eqref{condition4}, \eqref{gcd3}
\end{array}
\right\} \textrm{.}
\end{eqnarray*}
Furthermore, $k_1\eta_5$ and $k_4\eta_4$  are coprime since $\eta_3\eta_5$ and $\eta_1\eta_4$ are coprime thus the Chinese remainder theorem gives
\begin{eqnarray*}
S_{k_1,k_4} & = & \# \left\{ \alpha_2 \in \mathbb{Z},
\begin{array}{l}
\alpha_2 \equiv a \imod{k_1 k_4 \eta_4 \eta_5} \\
\eqref{condition3}, \eqref{condition4}, \eqref{gcd3}
\end{array}
\right\} \textrm{,}
\end{eqnarray*}
for a certain integer $a$ coprime to $k_1 k_4 \eta_4 \eta_5$ since $\gcd(k_1k_4\eta_4\eta_5,\alpha_2) = 1$. A Möbius inversion yields
\begin{eqnarray*}
S_{k_1,k_4} & = & \sum_{k_2|\eta_1\eta_2\eta_3\eta_4\eta_5} \mu(k_2) \# \left\{ \alpha_2' \in \mathbb{Z},
\begin{array}{l}
k_2 \alpha_2' \equiv a \imod{k_1 k_4 \eta_4 \eta_5} \\
\eqref{condition3}, \eqref{condition4}
\end{array}
\right\} \\
& = & \sum_{\substack{k_2|\eta_1\eta_2\eta_3 \\ \gcd(k_2,k_1k_4\eta_4\eta_5) = 1}} \mu(k_2)
\# \left\{ \alpha_2' \in \mathbb{Z},
\begin{array}{l}
\alpha_2' \equiv k_2^{-1} a \imod{k_1 k_4 \eta_4 \eta_5} \\
\eqref{condition3}, \eqref{condition4}
\end{array}
\right\}\textrm{,}
\end{eqnarray*}
since $\gcd(k_1 k_4 \eta_4 \eta_5,a) = 1$. Using the elementary estimate
\begin{eqnarray*}
\# \left\{ n \in \mathbb{Z} \cap [t_1,t_2], n \equiv a \imod{q} \right\} & = & \frac{t_2 - t_1}{q} + O(1) \textrm{,}
\end{eqnarray*}
and the change of variable $u_2 \mapsto u_2 A_2/k_2$, we get
\begin{eqnarray*}
\# \left\{ \alpha_2' \in \mathbb{Z},
\begin{array}{l}
\alpha_2' \equiv k_2^{-1} a \imod{k_1 k_4 \eta_4 \eta_5} \\
\eqref{condition3}, \eqref{condition4}
\end{array}
\right\} & = & \frac{A_2}{k_2 k_1 k_4 \eta_4 \eta_5} g_1 \left( \frac{\eta_7}{Y_7},\frac{\eta_6}{Y_6} \right) + O(1) \textrm{.}
\end{eqnarray*}
We see that the main term of $N(\boldsymbol{\eta}',B)$ is equal to
\begin{eqnarray*}
\frac{A_2}{\eta_4 \eta_5} g_1 \left( \frac{\eta_7}{Y_7},\frac{\eta_6}{Y_6} \right) \theta(\boldsymbol{\eta}') \textrm{,}
\end{eqnarray*}
where
\begin{eqnarray*}
\theta(\boldsymbol{\eta}') & = & \sum_{\substack{k_1|\eta_3 \\ \gcd(k_1,\eta_1\eta_2\eta_4\eta_6) = 1}} \frac{\mu(k_1)}{k_1} \sum_{\substack{k_4|\eta_1 \\ \gcd(k_4,\eta_2\eta_3\eta_5\eta_7) = 1}} \frac{\mu(k_4)}{k_4}
\sum_{\substack{k_2|\eta_1\eta_2\eta_3 \\ \gcd(k_2,k_1k_4\eta_4\eta_5) = 1}} \frac{\mu(k_2)}{k_2} \\
& = & \varphi^{\ast}(\eta_1\eta_2\eta_3\eta_4\eta_5)
\sum_{\substack{k_1|\eta_3 \\ \gcd(k_1,\eta_2\eta_6) = 1}} \frac{\mu(k_1)}{k_1\varphi^{\ast}(k_1\eta_5)}
\sum_{\substack{k_4|\eta_1 \\ \gcd(k_4,\eta_2\eta_7) = 1}} \frac{\mu(k_4)}{k_4\varphi^{\ast}(k_4\eta_4)} \textrm{.}
\end{eqnarray*}
We have removed $\eta_1\eta_4$ from the condition over $k_1$ and $\eta_3\eta_5$ from the condition over $k_4$ respectively because
$\gcd(\eta_3,\eta_1\eta_4) = 1$ and $\gcd(\eta_1,\eta_3\eta_5) = 1$. A straightforward calculation yields, for
$a,b,c \geq 1$,
\begin{eqnarray*}
\sum_{\substack{k|a \\ \gcd(k,c) = 1}} \frac{\mu(k)}{k\varphi^{\ast}(kb)} & = & \frac{\varphi^{\ast}(\gcd(a,b))}{\varphi^{\ast}(b)\varphi^{\ast}(\gcd(a,b,c))} \prod_{p|a, p \nmid bc} \left( 1 - \frac1{p-1} \right) \textrm{.}
\end{eqnarray*}
Therefore, we have obtained
\begin{eqnarray}
\label{theta}
\theta(\boldsymbol{\eta}') & = & \theta_1 (\boldsymbol{\eta},\eta_6)
\prod_{p|\eta_1, p \nmid \eta_2\eta_4\eta_7} \left( 1 - \frac1{p-1} \right) \textrm{,}
\end{eqnarray}
where $\theta_1 (\boldsymbol{\eta},\eta_6)$ denotes
\begin{eqnarray*}
\varphi^{\ast}(\eta_1\eta_2\eta_3\eta_4\eta_5)
\frac{\varphi^{\ast}(\gcd(\eta_1,\eta_4))}{\varphi^{\ast}(\eta_4)}
\frac{\varphi^{\ast}(\gcd(\eta_3,\eta_5))}{\varphi^{\ast}(\eta_5)} 
\prod_{p|\eta_3, p \nmid \eta_2\eta_5\eta_6} \left( 1 - \frac1{p-1} \right) \textrm{.}
\end{eqnarray*}
In addition, we see that the overall contribution of the error term is
\begin{eqnarray*}
\sum_{\boldsymbol{\eta},\eta_6,\eta_7} 2^{\omega(\eta_3)} 2^{\omega(\eta_1)} 2^{\omega(\eta_1\eta_2\eta_3)} & \ll & \sum_{\boldsymbol{\eta}} 2^{\omega(\eta_3)} 2^{\omega(\eta_1)} 2^{\omega(\eta_1\eta_2\eta_3)} Y_6 Y_7 \\
& = & \sum_{\boldsymbol{\eta}} 2^{\omega(\eta_3)} 2^{\omega(\eta_1)} 2^{\omega(\eta_1\eta_2\eta_3)} \frac{B}{\boldsymbol{\eta}^{(2,2,2,1,1)}} \\
& \ll & B \log(B)^2 \textrm{,}
\end{eqnarray*}
where we have summed over $\eta_6$ and $\eta_7$ using respectively the height conditions \eqref{condition1} and \eqref{condition2}. This completes the proof of lemma \ref{lemma inter}.

\subsection{Summation over $\eta_7$}

To carry out the summations over $\eta_7$ and $\eta_6$, we let
\begin{eqnarray}
\label{V}
\mathcal{V} & = & \left\{\boldsymbol{\eta} \in \mathbb{Z}_{>0}^5, Y_6 \geq 1, Y_7 \geq 1 \right\} \textrm{,}
\end{eqnarray}
and we assume that $\boldsymbol{\eta} \in \mathcal{V}$ is fixed and is subject to the coprimality conditions \eqref{gcd6} and \eqref{gcd7}. Our next task is to sum over $\eta_7$, that is why we have isolated $\eta_7$ in $\theta(\boldsymbol{\eta}')$. Let us define
\begin{eqnarray*}
\mathcal{N} & = & \{ (\eta_1,\eta_2,\eta_4) \in \mathbb{Z}_{>0}^3, 2 \nmid \eta_1 \textrm{ or } 2|\eta_2\eta_4 \} \textrm{.}
\end{eqnarray*}
It is plain to see that if
$(\eta_1,\eta_2,\eta_4) \in \mathcal{N}$ or $2|\eta_7$ then
\begin{eqnarray*}
\prod_{p|\eta_1, p \nmid \eta_2\eta_4\eta_7} \left( 1 - \frac1{p-1} \right) & = &
\prod_{\substack{p|\eta_1, p \nmid \eta_2\eta_4\eta_7 \\ p \neq 2}} \left( 1 - \frac1{p-1} \right) \textrm{,}
\end{eqnarray*}
and this product is equal to $0$ otherwise. Furthermore, since $\eta_2 \eta_4$ and $\eta_7$ are coprime, we see that
\begin{eqnarray*}
\prod_{\substack{p|\eta_1, p \nmid \eta_2\eta_4\eta_7 \\ p \neq 2}} \left( 1 - \frac1{p-1} \right) & = & \frac{\varphi^{\circ}(\eta_1)}{\varphi^{\circ}(\gcd(\eta_1,\eta_2\eta_4))\varphi^{\circ}(\gcd(\eta_1,\eta_7))} \textrm{.}
\end{eqnarray*}
We need to treat two cases separately depending on whether $(\eta_1,\eta_2,\eta_4) \in \mathcal{N}$ or
$(\eta_1,\eta_2,\eta_4) \notin \mathcal{N}$ (note that, in the latter case, the main term of $N(\boldsymbol{\eta}',B)$ vanishes if
$2 \nmid \eta_7$). For fixed $\eta_6$ satisfying the height condition \eqref{condition1} and the coprimality condition \eqref{gcd5}, we call $N(\boldsymbol{\eta},\eta_6,B)$ the sum of the main term of $N(\boldsymbol{\eta}',B)$ over $\eta_7$, $\eta_7$ being subject to the height condition \eqref{condition2} and to the coprimality condition \eqref{gcd4}. We also use $N_1(\boldsymbol{\eta},\eta_6,B)$ and $N_2(\boldsymbol{\eta},\eta_6,B)$ to denote the sums over $\eta_7$ respectively for $(\eta_1,\eta_2,\eta_4) \in \mathcal{N}$ and $(\eta_1,\eta_2,\eta_4) \notin \mathcal{N}$. We now proceed to prove the following lemma.

\begin{lemma}
\label{sum eta_7}
We have the estimate
\begin{eqnarray*}
N(\boldsymbol{\eta},\eta_6,B) & = & \frac{A_2Y_7}{\eta_4 \eta_5} g_2 \left( \frac{\eta_6}{Y_6}; \boldsymbol{\eta},B \right) \theta_1'(\boldsymbol{\eta}) \theta_2'(\boldsymbol{\eta},\eta_6) + R(\boldsymbol{\eta},\eta_6,B) \textrm{,}
\end{eqnarray*}
where $\theta_1'(\boldsymbol{\eta})$ and $\theta_2'(\boldsymbol{\eta},\eta_6)$ are arithmetic functions defined in \eqref{theta_1'} and \eqref{theta_2'} and
\begin{eqnarray*}
\sum_{\boldsymbol{\eta},\eta_6} R(\boldsymbol{\eta},\eta_6,B) & \ll & B \log(B)^4 \textrm{.}
\end{eqnarray*}
\end{lemma}

First, we estimate the contribution of $N_1(\boldsymbol{\eta},\eta_6,B)$. For this, we make use of the first estimate of lemma \ref{arithmetic preliminary} to deduce that for any fixed $0 < \delta \leq 1$, we have
\begin{eqnarray*}
N_1(\boldsymbol{\eta},\eta_6,B) & = & \frac{A_2Y_7}{\eta_4 \eta_5} g_2 \left( \frac{\eta_6}{Y_6}; \boldsymbol{\eta},B \right)
\theta_1 (\boldsymbol{\eta},\eta_6) \frac{\varphi^{\circ}(\eta_1)}{\varphi^{\circ}(\gcd(\eta_1,\eta_2\eta_4))} \Psi(\eta_1,\eta_2\eta_3\eta_4\eta_5\eta_6) \\
& & + O \left( \frac{A_2}{\eta_4 \eta_5} Y_7^{\delta} \sigma_{- \delta}(\eta_1\eta_2\eta_3\eta_4\eta_5\eta_6)
\sup_{t_7 Y_7 \geq 1} g_1 \left( t_7, \frac{\eta_6}{Y_6} \right) \right) \textrm{.}
\end{eqnarray*}
To estimate the overall contribution of this error term, we use the bound of lemma \ref{bounds} for $g_1$ and we choose
$\delta = 1/4$. The average order of $\sigma_{- 1/4}$ is $O(1)$ so we see that this contribution is
\begin{eqnarray*}
\sum_{\boldsymbol{\eta},\eta_6} \sigma_{- 1/4}(\eta_1\eta_2\eta_3\eta_4\eta_5\eta_6)
\frac{A_2 Y_6^{1/2} Y_7^{3/4}}{\eta_4 \eta_5 \eta_6^{1/2}} & \ll &
\sum_{\boldsymbol{\eta}} \sigma_{- 1/4}(\eta_1\eta_2\eta_3\eta_4\eta_5) \frac{A_2 Y_6 Y_7^{3/4}}{\eta_4 \eta_5} \\
& \ll & \sum_{\eta_1,\eta_2,\eta_3,\eta_5} \sigma_{- 1/4}(\eta_1\eta_2\eta_3\eta_5) \frac{B}{\boldsymbol{\eta}^{(1,1,1,0,1)}} \\
& \ll & B \log(B)^4 \textrm{,}
\end{eqnarray*}
where we have summed over $\eta_6$ and $\eta_4$ using respectively the conditions \eqref{condition1} and $Y_7 \geq 1$. Concerning the main term, we have
\begin{eqnarray*}
\Psi(\eta_1,\eta_2\eta_3\eta_4\eta_5\eta_6) & = & \varphi^{\ast}(\eta_2\eta_3\eta_4\eta_5\eta_6) \frac{\varphi^{\flat}(\eta_1)}{\varphi^{\flat}(\gcd(\eta_1,\eta_2\eta_4))} \textrm{,}
\end{eqnarray*}
and since $(\eta_1,\eta_2,\eta_4) \in \mathcal{N}$, we also have
\begin{eqnarray*}
\frac{\varphi^{\circ}(\eta_1)}{\varphi^{\circ}(\gcd(\eta_1,\eta_2\eta_4))}
\frac{\varphi^{\flat}(\eta_1)}{\varphi^{\flat}(\gcd(\eta_1,\eta_2\eta_4))} & = &
\frac{\varphi^{\ast}(\eta_1)}{\varphi^{\ast}(\gcd(\eta_1,\eta_2\eta_4))} \textrm{.}
\end{eqnarray*}
These equalities and a short calculation prove that
\begin{eqnarray*}
& & \theta_1(\boldsymbol{\eta},\eta_6) \frac{\varphi^{\circ}(\eta_1)}{\varphi^{\circ}(\gcd(\eta_1,\eta_2\eta_4))}
\Psi(\eta_1,\eta_2\eta_3\eta_4\eta_5\eta_6)
\end{eqnarray*}
can be rewritten as $\theta_1'(\boldsymbol{\eta}) \theta_2'(\boldsymbol{\eta},\eta_6)$ for
\begin{eqnarray}
\label{theta_1'}
\ \ \ \ \theta_1'(\boldsymbol{\eta}) & = & \varphi^{\ast}(\eta_1\eta_2\eta_3\eta_4\eta_5) \varphi^{\ast}(\eta_2\eta_3\eta_4\eta_5)
\frac{\varphi^{\ast}(\eta_1\eta_2)}{\varphi^{\ast}(\eta_2\eta_4)}
\frac{\varphi^{\ast}(\gcd(\eta_3,\eta_5))}{\varphi^{\ast}(\eta_5)} \textrm{,} \\
\label{theta_2'}
\ \ \ \ \theta_2'(\boldsymbol{\eta},\eta_6) & =& \frac{\varphi^{\ast}(\eta_6)}{\varphi^{\ast}(\gcd(\eta_6,\eta_3))}
\prod_{p|\eta_3, p \nmid \eta_2\eta_5\eta_6} \left( 1 - \frac1{p-1} \right) \textrm{.}
\end{eqnarray}

We now turn to the estimation of $N_2(\boldsymbol{\eta},\eta_6,B)$. We only need to sum on the even $\eta_7$ and so, given the coprimality condition \eqref{gcd4}, $\eta_2\eta_3\eta_4\eta_5\eta_6$ is odd and thus we can make use of the first estimate of lemma \ref{arithmetic preliminary 2}. The error term is the same as the previous one and, in the main term, there are exactly two differences with the case of $N_1(\boldsymbol{\eta},\eta_6,B)$. The first is the factor $1/2$ and the second is the fact that here, since $(\eta_1,\eta_2,\eta_4) \notin \mathcal{N}$,
\begin{eqnarray*}
\frac{\varphi^{\circ}(\eta_1)}{\varphi^{\circ}(\gcd(\eta_1,\eta_2\eta_4))}
\frac{\varphi^{\flat}(\eta_1)}{\varphi^{\flat}(\gcd(\eta_1,\eta_2\eta_4))} & = & 2
\frac{\varphi^{\ast}(\eta_1)}{\varphi^{\ast}(\gcd(\eta_1,\eta_2\eta_4))} \textrm{,}
\end{eqnarray*}
and thus we find exactly the same main term, which completes the proof of lemma \ref{sum eta_7}.

\subsection{Summation over $\eta_6$}

We now proceed to sum over $\eta_6$. We set
\begin{eqnarray*}
\mathcal{M} & = & \{ (\eta_3,\eta_2,\eta_5) \in \mathbb{Z}_{>0}^3, 2 \nmid \eta_3 \textrm{ or } 2|\eta_2\eta_5 \} \textrm{.}
\end{eqnarray*}
As for the summation over $\eta_7$, it is clear that if $(\eta_3,\eta_2,\eta_5) \in \mathcal{M}$ or $2|\eta_6$ then
\begin{eqnarray*}
\prod_{p|\eta_3, p \nmid \eta_2\eta_5\eta_6} \left( 1 - \frac1{p-1} \right) & = &
\prod_{\substack{p|\eta_3, p \nmid \eta_2\eta_5\eta_6 \\ p \neq 2}} \left( 1 - \frac1{p-1} \right) \textrm{,}
\end{eqnarray*}
and this product is equal to $0$ otherwise. Furthermore, since $\eta_2 \eta_5$ and $\eta_6$ are coprime, we have
\begin{eqnarray*}
\prod_{\substack{p|\eta_3, p \nmid \eta_2\eta_5\eta_6 \\ p \neq 2}} \left( 1 - \frac1{p-1} \right) & = &
\frac{\varphi^{\circ}(\eta_3)}{\varphi^{\circ}(\gcd(\eta_3,\eta_2\eta_5))\varphi^{\circ}(\gcd(\eta_3,\eta_6))} \textrm{.}
\end{eqnarray*}
We need to treat two cases separately depending on whether $(\eta_3,\eta_2,\eta_5) \in \mathcal{M}$ or
$(\eta_3,\eta_2,\eta_5) \notin \mathcal{M}$ (note that, in the latter case, the main term of $N(\boldsymbol{\eta},\eta_6,B)$ vanishes if $2 \nmid \eta_6$). Let $\mathbf{N}(\boldsymbol{\eta},B)$ be the sum of the main term of $N(\boldsymbol{\eta},\eta_6,B)$ over $\eta_6$, $\eta_6$ satisfying the height condition \eqref{condition1} and the coprimality condition \eqref{gcd5} and let also $\mathbf{N}_1(\boldsymbol{\eta},B)$ and $\mathbf{N}_2(\boldsymbol{\eta},B)$ be the sums over $\eta_6$ respectively for $(\eta_3,\eta_2,\eta_5) \in \mathcal{M}$ and
$(\eta_3,\eta_2,\eta_5) \notin \mathcal{M}$.

\begin{lemma}
\label{sum eta_6}
We have the estimate
\begin{eqnarray*}
\mathbf{N}(\boldsymbol{\eta},B) & = & \zeta(2)^{-1} \frac{B}{\boldsymbol{\eta}^{(1,1,1,1,1)}} g_3(\boldsymbol{\eta},B) \Theta(\boldsymbol{\eta}) + \mathbf{R}(\boldsymbol{\eta},B) \textrm{,}
\end{eqnarray*}
where
\begin{eqnarray*}
\Theta(\boldsymbol{\eta}) & = & \frac{\varphi^{\ast}(\eta_1\eta_2\eta_3\eta_4\eta_5)}{\varphi^{\dag}(\eta_1\eta_2\eta_3\eta_4\eta_5)} \varphi^{\ast}(\eta_2\eta_3\eta_4\eta_5) \varphi^{\ast}(\eta_1\eta_2\eta_4\eta_5)
\frac{\varphi^{\ast}(\eta_1\eta_2)}{\varphi^{\ast}(\eta_2\eta_4)} \frac{\varphi^{\ast}(\eta_2\eta_3)}{\varphi^{\ast}(\eta_2\eta_5)} \textrm{,}
\end{eqnarray*}
and 
\begin{eqnarray*}
\sum_{\boldsymbol{\eta}} \mathbf{R}(\boldsymbol{\eta},B) & \ll & B \log(B)^4 \textrm{.}
\end{eqnarray*}
\end{lemma}

We first treat the contribution of $\mathbf{N}_1(\boldsymbol{\eta},B)$. For this, we make use of the second estimate of lemma
\ref{arithmetic preliminary} to deduce that for any fixed $0 < \delta \leq 1$, we have
\begin{eqnarray*}
\mathbf{N}_1(\boldsymbol{\eta},B) & = & \frac{A_2Y_7Y_6}{\eta_4 \eta_5} g_3 \left( \boldsymbol{\eta},B \right)
\theta_1' (\boldsymbol{\eta}) \frac{\varphi^{\circ}(\eta_3)}{\varphi^{\circ}(\gcd(\eta_3,\eta_2\eta_5))}
\Psi'(\eta_3,\eta_1\eta_2\eta_4\eta_5) \\
& & + O \left( \frac{A_2Y_7}{\eta_4 \eta_5} Y_6^{\delta} \sigma_{- \delta}(\eta_1\eta_2\eta_4\eta_5)
\sup_{t_6 Y_6 \geq 1} g_2 \left( t_6; \boldsymbol{\eta},B \right) \right) \textrm{.}
\end{eqnarray*}
To estimate the overall contribution of the error term, we use the bound of lemma \ref{bounds} for $g_2$ and we choose $\delta = 1/4$. Since the average order of $\sigma_{- 1/4}$ is $O(1)$, we obtain that this contribution is
\begin{eqnarray*}
\sum_{\boldsymbol{\eta}} \sigma_{- 1/4}(\eta_1\eta_2\eta_4\eta_5) \frac{A_2 Y_7 Y_6^{3/4}}{\eta_4 \eta_5} & \ll &
\sum_{\eta_1,\eta_2,\eta_3,\eta_4} \sigma_{- 1/4}(\eta_1\eta_2\eta_4) \frac{B}{\boldsymbol{\eta}^{(1,1,1,1,0)}} \\
& \ll & B \log(B)^4 \textrm{,}
\end{eqnarray*}
where we have summed over $\eta_5$ using the condition $Y_6 \geq 1$. Let us turn to the main term. First, note that
\begin{eqnarray*}
\frac{A_2Y_7Y_6}{\eta_4 \eta_5} & = & \frac{B}{\boldsymbol{\eta}^{(1,1,1,1,1)}} \textrm{.}
\end{eqnarray*}
In addition, we have
\begin{eqnarray*}
\Psi'(\eta_3,\eta_1\eta_2\eta_4\eta_5) & = & \varphi^{\ast}(\eta_1\eta_2\eta_4\eta_5)
\frac{\varphi^{\flat}(\eta_3)}{\varphi^{\flat}(\gcd(\eta_3,\eta_2\eta_5))}
\frac{\zeta(2)^{-1}}{\varphi^{\dag}(\eta_1\eta_2\eta_3\eta_4\eta_5)} \textrm{,}
\end{eqnarray*}
and since $(\eta_3,\eta_2,\eta_5) \in \mathcal{M}$, we also have
\begin{eqnarray*}
\frac{\varphi^{\circ}(\eta_3)}{\varphi^{\circ}(\gcd(\eta_3,\eta_2\eta_5))}
\frac{\varphi^{\flat}(\eta_3)}{\varphi^{\flat}(\gcd(\eta_3,\eta_2\eta_5))} & = &
\frac{\varphi^{\ast}(\eta_3)}{\varphi^{\ast}(\gcd(\eta_3,\eta_2\eta_5))} \textrm{.}
\end{eqnarray*}
An easy calculation finally yields
\begin{eqnarray*}
\theta_1' (\boldsymbol{\eta}) \frac{\varphi^{\circ}(\eta_3)}{\varphi^{\circ}(\gcd(\eta_3,\eta_2\eta_5))}
\Psi'(\eta_3,\eta_1\eta_2\eta_4\eta_5) & = & \zeta(2)^{-1} \Theta(\boldsymbol{\eta}) \textrm{.}
\end{eqnarray*}

We now deal with the estimation of $\mathbf{N}_2(\boldsymbol{\eta},B)$. We only need to sum on the even $\eta_6$ and so, given the coprimality condition \eqref{gcd5}, $\eta_1\eta_2\eta_4\eta_5$ is odd and moreover since
$(\eta_3,\eta_2,\eta_5) \notin \mathcal{M}$, we have $2|\eta_3$ and thus we can make use of the second estimate of lemma \ref{arithmetic preliminary 2}. The error term is the same as the previous one and, in the main term, there are exactly two differences with the case of $\mathbf{N}_1(\boldsymbol{\eta},B)$. The first is the factor $1/2$ and the second is that here, since
$(\eta_3,\eta_2,\eta_5) \notin \mathcal{N}$,
\begin{eqnarray*}
\frac{\varphi^{\circ}(\eta_3)}{\varphi^{\circ}(\gcd(\eta_3,\eta_2\eta_5))}
\frac{\varphi^{\flat}(\eta_3)}{\varphi^{\flat}(\gcd(\eta_3,\eta_2\eta_5))} & = & 2
\frac{\varphi^{\ast}(\eta_3)}{\varphi^{\ast}(\gcd(\eta_3,\eta_2\eta_5))} \textrm{,}
\end{eqnarray*}
and we finally obtain the same main term, which concludes the proof of lemma \ref{sum eta_6}.

\subsection{Conclusion}

The aim of the following lemma is to replace the conditions $t_6 Y_6 \geq 1$ and $t_7 Y_7 \geq 1$ in the integral \eqref{g_3} defining $g_3$ in the main term of $\mathbf{N}(\boldsymbol{\eta},B)$ in lemma \ref{sum eta_6} respectively by $t_6 > 0$ and $t_7 > 0$. For short, we introduce the notation
\begin{eqnarray*}
D_h & = & \left\{ (u_2, t_7, t_6) \in \mathbb{R}^3, t_6,t_7>0, h(u_2,t_7,t_6) \leq 1 \right\} \textrm{.}
\end{eqnarray*}

\begin{lemma}
For $Z_6, Z_7 > 0$, we have
\begin{eqnarray}
\label{1}
\meas \{ (u_2, t_7, t_6) \in D_h, t_6 Z_6 < 1 \} & \ll & Z_6^{-1/2} \textrm{,} \\
\label{2}
\meas \{ (u_2, t_7, t_6) \in D_h, t_7 Z_7 < 1 \} & \ll & Z_7^{-1/2} \textrm{.}
\end{eqnarray}
\end{lemma}

\begin{proof}
These two bounds follow from the bound of lemma \ref{bounds} for $g_1$ and the fact that $h(u_2,t_7,t_6) \leq 1$ implies
$t_6,t_7 \leq 1$.
\end{proof}

Making use of the bound \eqref{1}, we see that replacing the condition $t_6 Y_6 \geq 1$ in the integral defining $g_3$ in the main term of $\mathbf{N}(\boldsymbol{\eta},B)$ in lemma \ref{sum eta_6} by the condition $t_6 > 0$ creates an error term whose overall contribution is
\begin{eqnarray*}
\sum_{\boldsymbol{\eta}}  \frac{A_2 Y_7 Y_6^{1/2}}{\eta_4 \eta_5} & \ll &
\sum_{\eta_1,\eta_2,\eta_3,\eta_4} \frac{B}{\boldsymbol{\eta}^{(1,1,1,1,0)}} \\
& \ll & B \log(B)^4 \textrm{,}
\end{eqnarray*}
where we have summed over $\eta_5$ using the condition $Y_6 \geq 1$. The bound \eqref{2} shows that the same conclusion holds for the condition
$t_7 Y_7 \geq 1$. Recalling the equality \eqref{omega}, we finally see that we can replace $g_3(\boldsymbol{\eta},B)$ in the main term of $\mathbf{N}(\boldsymbol{\eta},B)$ in lemma \ref{sum eta_6} by
\begin{eqnarray*}
\int \int \int_{t_6,t_7>0, h(u_2,t_7,t_6) \leq 1} \D u_2 \D t_7 \D t_6 & = & \frac{\omega_{\infty}}{4} \textrm{.}
\end{eqnarray*}
Redefine $\Theta$ as being equal to zero if the remaining coprimality conditions \eqref{gcd6} and \eqref{gcd7} are not satisfied. Using lemma \ref{T}, we obtain the following result.

\begin{lemma}
\label{final lemma}
We have the estimate
\begin{eqnarray*}
N_{U,H}(B) & = & \zeta(2)^{-1} \frac{\omega_{\infty}}{2} B \sum_{\boldsymbol{\eta} \in \mathcal{V}}
\frac{\Theta(\boldsymbol{\eta})}{\boldsymbol{\eta}^{(1,1,1,1,1)}} + O \left( B \log(B)^4 \right) \textrm{,}
\end{eqnarray*}
where $\mathcal{V}$ is defined in \eqref{V}.
\end{lemma}

Let us introduce the generalized Möbius function $\boldsymbol{\mu}$ defined for $(n_1, \dots, n_5) \in \mathbb{Z}_{>0}^5$ by
$\boldsymbol{\mu}(n_1, \dots, n_5) = \mu(n_1) \cdots \mu(n_5)$. We set $\mathbf{k} = (k_1,k_2,k_3,k_4,k_5)$ and we define, for
$s \in \mathbb{C}$ such that $\Re(s) > 1$,
\begin{eqnarray*}
F(s) & = & \sum_{\boldsymbol{\eta} \in \mathbb{Z}_{>0}^5}
\frac{\left|(\Theta \ast \boldsymbol{\mu})(\boldsymbol{\eta})\right|}{\eta_1^s \eta_2^s \eta_3^s \eta_4^s \eta_5^s} \\
& = & \prod_p \left( \sum_{\mathbf{k} \in \mathbb{Z}_{\geq 0}^5}
\frac{\left|(\Theta \ast \boldsymbol{\mu}) \left( p^{k_1},p^{k_2},p^{k_3},p^{k_4},p^{k_5} \right)\right|}
{p^{k_1 s}p^{k_2 s}p^{k_3 s}p^{k_4 s}p^{k_5 s}} \right) \textrm{.}
\end{eqnarray*}
It is easy to see that if $\mathbf{k} \notin \{0,1\}^5$ then
$(\Theta \ast \boldsymbol{\mu}) \left( p^{k_1},p^{k_2},p^{k_3},p^{k_4},p^{k_5} \right) = 0$ and moreover if exactly one of the $k_i$ is equal to $1$, then $(\Theta \ast \boldsymbol{\mu}) \left( p^{k_1},p^{k_2},p^{k_3},p^{k_4},p^{k_5} \right) \ll 1/p$, so the local factors $F_p$ of $F$ satisfy
\begin{eqnarray*}
F_p(s) & = & 1 + O \left( \frac1{p^{ \min \left( \Re(s)+1, 2 \Re(s) \right)}} \right) \textrm{.}
\end{eqnarray*}
This proves that $F$ actually converges in the half-plane $\Re(s) > 1/2$, which implies that $\Theta$ satifies the assumption of
\cite[Lemma $8$]{3A1}. Applying this lemma, we get
\begin{eqnarray}
\label{sum1}
\ \ \ \ \ \sum_{\boldsymbol{\eta} \in \mathcal{V}} \frac{\Theta(\boldsymbol{\eta})}{\boldsymbol{\eta}^{(1,1,1,1,1)}} & = & \alpha
\left( \sum_{\boldsymbol{\eta} \in \mathbb{Z}_{>0}^5} \frac{(\Theta \ast \boldsymbol{\mu})(\boldsymbol{\eta})}{\boldsymbol{\eta}^{(1,1,1,1,1)}} \right) \log(B)^5 + O \left( \log(B)^4 \right) \textrm{,}
\end{eqnarray}
where $\alpha$ is the volume of the polytope defined in $\mathbb{R}^5$ by $t_1,t_2,t_3,t_4,t_5 \geq 0$ and
\begin{eqnarray*}
t_1 + 2 t_2 + 3 t_3 + 2 t_5 & \leq & 1 \textrm{,} \\
3 t_1 + 2 t_2 + t_3 + 2 t_4 & \leq & 1 \textrm{.}
\end{eqnarray*}
A computation using Franz's additional \textit{Maple} package \cite{Convex} provides $\alpha = 1/2160$, that is to say
\begin{eqnarray}
\label{alpha}
\alpha & = & 2 \alpha(\widetilde{V}) \textrm{.}
\end{eqnarray}
Moreover,
\begin{eqnarray*}
\sum_{\boldsymbol{\eta} \in \mathbb{Z}_{>0}^5} \frac{(\Theta \ast \boldsymbol{\mu}) (\boldsymbol{\eta})}{\boldsymbol{\eta}^{(1,1,1,1,1)}} & = & \prod_p \left( \sum_{\mathbf{k} \in \mathbb{Z}_{\geq 0}^5}
\frac{(\Theta \ast \boldsymbol{\mu}) \left( p^{k_1},p^{k_2},p^{k_3},p^{k_4},p^{k_5} \right)}{p^{k_1}p^{k_2}p^{k_3}p^{k_4}p^{k_5}} \right) \\
& = & \prod_p \left( 1 - \frac1{p} \right)^5 \left( \sum_{\mathbf{k} \in \mathbb{Z}_{\geq 0}^5} 
\frac{\Theta \left( p^{k_1},p^{k_2},p^{k_3},p^{k_4},p^{k_5} \right)}{p^{k_1}p^{k_2}p^{k_3}p^{k_4}p^{k_5}} \right) \textrm{.}
\end{eqnarray*}
The remaining coprimality conditions greatly simplify the calculation and we obtain
\begin{eqnarray*}
\sum_{\mathbf{k} \in \mathbb{Z}_{\geq 0}^5}
\frac{\Theta \left( p^{k_1},p^{k_2},p^{k_3},p^{k_4},p^{k_5} \right)}{p^{k_1}p^{k_2}p^{k_3}p^{k_4}p^{k_5}} & = &
\left( 1 - \frac1{p^2} \right)^{-1} \left( 1 - \frac1{p} \right) \left( 1 + \frac{6}{p} + \frac1{p^2} \right) \textrm{,}
\end{eqnarray*}
which gives
\begin{eqnarray}
\label{sum2}
\sum_{\boldsymbol{\eta} \in \mathbb{Z}_{>0}^5} \frac{(\Theta \ast \boldsymbol{\mu}) (\boldsymbol{\eta})}{\boldsymbol{\eta}^{(1,1,1,1,1)}} & = & \zeta(2) \prod_p \left( 1 - \frac1{p} \right)^6 \omega_p \textrm{.}
\end{eqnarray}
We complete the proof of theorem \ref{Manin} putting together the equalities \eqref{sum1}, \eqref{alpha}, \eqref{sum2} and lemma \ref{final lemma}.

\bibliographystyle{is-alpha}
\bibliography{biblio}

\begin{thebibliography}{FMT89}

\bibitem[BB07]{MR2320172}
R.~de~la Bret{\`e}che and T.~D. Browning.
\newblock On {M}anin's conjecture for singular del {P}ezzo surfaces of degree
  4. {I}.
\newblock {\em Michigan Math. J.}, 55\penalty0 (1):\penalty0 51--80, 2007.

\bibitem[BD09]{MR2543667}
T.~D. Browning and U.~Derenthal.
\newblock Manin's conjecture for a quartic del {P}ezzo surface with {$A_4$}
  singularity.
\newblock {\em Ann. Inst. Fourier (Grenoble)}, 59\penalty0 (3):\penalty0
  1231--1265, 2009.

\bibitem[BT98]{MR1620682}
V.~V. Batyrev and Y.~Tschinkel.
\newblock Manin's conjecture for toric varieties.
\newblock {\em J. Algebraic Geom.}, 7\penalty0 (1):\penalty0 15--53, 1998.

\bibitem[CLT02]{MR1906155}
A.~Chambert-Loir and Y.~Tschinkel.
\newblock On the distribution of points of bounded height on equivariant
  compactifications of vector groups.
\newblock {\em Invent. Math.}, 148\penalty0 (2):\penalty0 421--452, 2002.

\bibitem[CT88]{MR940430}
D.~F. Coray and M.~A. Tsfasman.
\newblock Arithmetic on singular {D}el {P}ezzo surfaces.
\newblock {\em Proc. London Math. Soc. (3)}, 57\penalty0 (1):\penalty0 25--87,
  1988.

\bibitem[Der06a]{Der-th}
U.~Derenthal.
\newblock {\em Geometry of universal torsors}.
\newblock PhD thesis, {G}eorg-{A}ugust-{U}ni\-versität {G}öttingen, 2006.

\bibitem[Der06b]{D-hyper}
U.~Derenthal.
\newblock Singular del {P}ezzo surfaces whose universal torsors are
  hypersurfaces.
\newblock {\em arXiv:math/0604194v1}, 2006.

\bibitem[Der07]{MR2318651}
U.~Derenthal.
\newblock On a constant arising in {M}anin's conjecture for del {P}ezzo
  surfaces.
\newblock {\em Math. Res. Lett.}, 14\penalty0 (3):\penalty0 481--489, 2007.

\bibitem[Der09]{MR2520770}
U.~Derenthal.
\newblock Counting integral points on universal torsors.
\newblock {\em Int. Math. Res. Not. IMRN}, \penalty0 (14):\penalty0 2648--2699,
  2009.

\bibitem[DL10]{DL-equi}
U.~Derenthal and D.~Loughran.
\newblock Singular del {P}ezzo surfaces that are equivariant compactifications.
\newblock {\em Zapiski Nauchnykh Seminarov (POMI)}, 377:\penalty0 26--43, 2010.

\bibitem[DT07]{MR2290499}
U.~Derenthal and Y.~Tschinkel.
\newblock Universal torsors over del {P}ezzo surfaces and rational points.
\newblock In {\em Equidistribution in number theory, an introduction}, volume
  237 of {\em NATO Sci. Ser. II Math. Phys. Chem.}, pages 169--196. Springer,
  Dordrecht, 2007.

\bibitem[FMT89]{MR974910}
J.~Franke, Y.~I. Manin, and Y.~Tschinkel.
\newblock Rational points of bounded height on {F}ano varieties.
\newblock {\em Invent. Math.}, 95\penalty0 (2):\penalty0 421--435, 1989.

\bibitem[Fra09]{Convex}
M.~Franz.
\newblock Convex - a {M}aple package for convex geometry, version 1.1, 2009.

\bibitem[LB10]{3A1}
P.~Le~Boudec.
\newblock {M}anin's conjecture for two quartic del {P}ezzo surfaces with {$3
  \mathbf{A}_1$} and {$\mathbf{A}_1 + \mathbf{A}_2$} singularity types.
\newblock {\em Acta Arith., to appear, arXiv:1006.0691v2}, 2010.

\bibitem[Lou10]{Loughran}
D.~Loughran.
\newblock Manin's conjecture for a singular sextic del {P}ezzo surface.
\newblock {\em J. Théor. Nombres Bordeaux}, 22\penalty0 (3):\penalty0 675--701,
  2010.

\bibitem[Pey95]{MR1340296}
E.~Peyre.
\newblock Hauteurs et mesures de {T}amagawa sur les vari\'et\'es de {F}ano.
\newblock {\em Duke Math. J.}, 79\penalty0 (1):\penalty0 101--218, 1995.

\end{thebibliography}

\end{document}